\documentclass{article}
\usepackage{amssymb}
\usepackage{amsmath}
\usepackage{makeidx}

\newtheorem{theorem}{Theorem}

\newtheorem{corollary}[theorem]{Corollary}

\newtheorem{definition}[theorem]{Definition}
\newtheorem{example}[theorem]{Example}

\newtheorem{remark}[theorem]{Remark}

\newenvironment{proof}[1][Proof]{\noindent\textbf{#1.} }{\ \rule{0.5em}{0.5em}}
\input{tcilatex}
\begin{document}

\title{A distinguished Riemannian geometrization for quadratic Hamiltonians
of polymomenta}
\author{Alexandru Oan\u{a} and Mircea Neagu}
\date{}
\maketitle

\begin{abstract}
In this paper we construct a distinguished Riemannian geometrization on the
dual $1$-jet space $J^{1\ast }(\mathcal{T},M)$ for the multi-time quadratic
Hamiltonian function%
\begin{equation*}
H=h_{ab}(t)g^{ij}(t,x)p_{i}^{a}p_{j}^{b}+U_{(a)}^{(i)}(t,x)p_{i}^{a}+%
\mathcal{F}(t,x).
\end{equation*}%
Our geometrization includes a nonlinear connection $N,$ a generalized Cartan
canonical $N$-linear connection $C\Gamma (N)$ (together with its local
d-torsions and d-curvatures), naturally provided by the given quadratic
Hamiltonian function depending on polymomenta.
\end{abstract}



\section{Short introduction}

In the last decades, numerous scientists were preoccupied by the
geometrization of Hamiltonians depending on polymomenta. In such a
perspective, we point out that the Hamiltonian geometrizations are achieved
in three distinct ways:

\begin{itemize}
\item[$\blacklozenge $] the \textit{multisymplectic Hamiltonian geometry} $-$
developed by Gotay, Isenberg, Marsden, Montgomery and their co-workers (see 
\cite{Gota+Isen+Mars}, \cite{Gota+Isen+Mars+Mont});

\item[$\blacklozenge $] the \textit{polysymplectic Hamiltonian geometry} $-$
elaborated by Giachetta, Mangiarotti and Sardanashvily (see \cite%
{Giac+Mang+Sard1}, \cite{Giac+Mang+Sard2});

\item[$\blacklozenge $] the \textit{De Donder-Weyl Hamiltonian geometry} $-$
studied by Kanatchikov (see \cite{Kana1}, \cite{Kana2}, \cite{Kana3}).
\end{itemize}

In such a geometrical context, the recent studies of Atanasiu and Neagu (see
the papers \cite{Atan-Neag2}, \cite{Atan-Neag0} and \cite{Atan+Neag1}) are
initiating the new way of distinguished Riemannian geometrization for
Hamiltonians depending on polymomenta, which is in fact a natural
"multi-time" extension of the already classical Hamiltonian geometry on
cotangent bundles synthesized in the Miron et al.'s book \cite%
{Miro+Hrim+Shim+Saba}. Note that our distinguished Riemannian geometrization
for Hamiltonians depending on polymomenta is different one by all three
Hamiltonian geometrizations from above (multisymplectic, polysymplectic and
De Donder-Weyl).

\section{Metrical multi-time Hamilton spaces}

Let us consider that $h=\left( h_{ab}\left( t\right) \right) $ is a
semi-Riemannian metric on the "multi-time" (\textit{temporal}) manifold $%
\mathcal{T}^{m}$, where $m=\dim \mathcal{T}$. Let $g=\left(
g^{ij}(t^{c},x^{k},p_{k}^{c})\right) $ be a symmetric d-tensor on the dual $%
1 $-jet space $E^{\ast }=J^{1\ast }(\mathcal{T},M^{n})$, which has the rank $%
n=\dim M$ and a constant signature. At the same time, let us consider a
smooth multi-time Hamiltonian function%
\begin{equation*}
E^{\ast }\ni (t^{a},x^{i},p_{i}^{a})\rightarrow H(t^{a},x^{i},p_{i}^{a})\in 
\mathbb{R},
\end{equation*}%
which yields the \textit{fundamental vertical metrical d-tensor}%
\begin{equation*}
G_{(a)(b)}^{(i)(j)}={\dfrac{1}{2}}{\frac{\partial ^{2}H}{\partial
p_{i}^{a}\partial p_{j}^{b}},}
\end{equation*}%
where $a,b=1,..,m$ and $i,j=1,...,n.$

\begin{definition}
A multi-time Hamiltonian function $H:E^{\ast }\rightarrow \mathbb{R},$
having the fundamental vertical metrical d-tensor of the form 
\begin{equation*}
G_{(a)(b)}^{(i)(j)}(t^{c},x^{k},p_{k}^{c})={\frac{1}{2}}{\frac{\partial ^{2}H%
}{\partial p_{i}^{a}\partial p_{j}^{b}}}%
=h_{ab}(t^{c})g^{ij}(t^{c},x^{k},p_{k}^{c}),
\end{equation*}%
is called a \textbf{Kronecker }$h$\textbf{-regular multi-time Hamiltonian
function}.
\end{definition}

In such a context, we can introduce the following important geometrical
concept:

\begin{definition}
A pair $MH_{m}^{n}=(E^{\ast }=J^{1\ast }(\mathcal{T},M),H),$ where $m=\dim 
\mathcal{T}$ and $n=\dim M,$ consisting of the dual $1$-jet space and a
Kronecker $h$-regular multi-time Hamiltonian function $H:E^{\ast
}\rightarrow \mathbb{R},$ is called a \textbf{multi-time Hamilton space}.
\end{definition}

\begin{remark}
In the particular case $(\mathcal{T},h)=(\mathbb{R},\delta ),$ a \textbf{%
"single-time" Hamilton space} will be also called a \textbf{relativistic
rheonomic Hamilton space}\textit{\ and it will be denoted by }$%
RRH^{n}=(J^{1\ast }(\mathbb{R},M),H)$.
\end{remark}

\begin{example}
Let us consider the Kronecker $h$-regular multi-time Hamiltonian function $%
H_{1}:E^{\ast }\rightarrow \mathbb{R}$ given by%
\begin{equation}
H_{1}=\frac{1}{mc}h_{ab}(t)\varphi ^{ij}(x)p_{i}^{a}p_{j}^{b},  \label{G}
\end{equation}%
where $h_{ab}(t)$ ($\varphi _{ij}(x)$, respectively) is a semi-Riemannian
metric on the temporal (spatial, respectively) manifold $\mathcal{T}$ ($M$,
respectively) having the physical meaning of \textbf{gravitational potentials%
}, and $m$ and $c$ are the known constants from Theoretical Physics
representing the \textbf{mass of the test body} and the \textbf{speed of
light}. Then, the multi-time Hamilton space $\mathcal{G}MH_{m}^{n}=(E^{\ast
},H_{1})$ is called the \textbf{multi-time Hamilton space of the
gravitational field}.
\end{example}

\begin{example}
If we consider on $E^{\ast }$ a symmetric d-tensor field $g^{ij}(t,x)$,
having the rank $n$ and a constant signature, we can define the Kronecker $h$%
-regular multi-time Hamiltonian function $H_{2}:E^{\ast }\rightarrow \mathbb{%
R},$ by setting%
\begin{equation}
H_{2}=h_{ab}(t)g^{ij}(t,x)p_{i}^{a}p_{j}^{b}+U_{(a)}^{(i)}(t,x)p_{i}^{a}+%
\mathcal{F}(t,x),  \label{NED}
\end{equation}%
where $U_{(a)}^{(i)}(t,x)$ is a d-tensor field on $E^{\ast },$ and $\mathcal{%
F}(t,x)$ is a function on $E^{\ast }$. Then, the multi-time Hamilton space $%
\mathcal{NED}MH_{m}^{n}=(E^{\ast },H_{2})$ is called the \textbf{%
non-a\-u\-to\-no\-mous multi-time Hamilton space of electrodynamics}. The
dynamical character of the gravitational potentials $g_{ij}(t,x)$ (i.e., the
dependence on the temporal coordinates $t^{c}$) motivated us to use the word 
\textbf{"non-autonomous".}
\end{example}

An important role for the subsequent development of our distinguished
Riemannian geometrical theory for multi-time Hamilton spaces is
re\-pre\-sen\-ted by the following result (proved in the paper \cite%
{Atan-Neag2}):

\begin{theorem}
\label{thchar} If we have $m=\dim \mathcal{T}\geq 2,$ then the following
statements are equivalent:

\emph{(i)} $H$ is a Kronecker $h$-regular multi-time Hamiltonian function on 
$E^{\ast }$.

\emph{(ii)} The multi-time Hamiltonian function $H$ reduces to a multi-time
Hamiltonian function of non-autonomous electrodynamic type. In other words
we have%
\begin{equation}
H=h_{ab}(t)g^{ij}(t,x)p_{i}^{a}p_{j}^{b}+U_{(a)}^{(i)}(t,x)p_{i}^{a}+%
\mathcal{F}(t,x).  \label{NEDTH}
\end{equation}
\end{theorem}

\begin{corollary}
The \textit{fundamental vertical metrical d-tensor of a }Kronecker $h$%
-regular multi-time Hamiltonian function $H$ has the form%
\begin{equation}
G_{(a)(b)}^{(i)(j)}={\frac{1}{2}}{\frac{\partial ^{2}H}{\partial
p_{i}^{a}\partial p_{j}^{b}}}=\left\{ 
\begin{array}{ll}
h_{11}(t)g^{ij}(t,x^{k},p_{k}^{1}), & m=\dim \mathcal{T}=1\medskip \\ 
h_{ab}(t^{c})g^{ij}(t^{c},x^{k}), & m=\dim \mathcal{T}\geq 2.%
\end{array}%
\right.  \label{FVDT}
\end{equation}
\end{corollary}

We recall that the transformations of coordinates on the dual $1$-jet vector
bundle $J^{1\ast }(\mathcal{T},M)$ are given by%
\begin{equation*}
\begin{array}{ccc}
\widetilde{t}^{a}=\widetilde{t}^{a}\left( t^{b}\right) , & \widetilde{x}^{i}=%
\widetilde{x}^{i}\left( x^{j}\right) , & \widetilde{p}_{i}^{a}=\dfrac{%
\partial x^{j}}{\partial \widetilde{x}^{i}}\dfrac{\partial \widetilde{t}^{a}%
}{\partial t^{b}}p_{j}^{b},%
\end{array}%
\end{equation*}%
where $\det \left( \partial \widetilde{t}^{a}/\partial t^{b}\right) \neq 0$
and $\det \left( \partial \widetilde{x}^{i}/\partial x^{j}\right) \neq 0.$
In this context, let us introduce the following important geometrical
concept:

\begin{definition}
A pair of local functions on $E^{\ast }=J^{1\ast }(\mathcal{T},M),$ denoted
by 
\begin{equation*}
N=\left( \underset{1}{N}\overset{\left( a\right) }{_{\left( i\right) b}},\ 
\underset{2}{N}\overset{\left( a\right) }{_{\left( i\right) j}}\right) ,
\end{equation*}%
whose local components obey the transformation rules%
\begin{equation*}
\begin{array}{l}
\underset{1}{\widetilde{N}}\overset{\left( b\right) }{_{\left( j\right) c}}%
\dfrac{\partial \widetilde{t}^{c}}{\partial t^{a}}=\underset{1}{N}\overset{%
\left( c\right) }{_{\left( k\right) a}}\dfrac{\partial \widetilde{t}^{b}}{%
\partial t^{c}}\dfrac{\partial x^{k}}{\partial \widetilde{x}^{j}}-\dfrac{%
\partial \widetilde{p}_{j}^{b}}{\partial t^{a}},\medskip \\ 
\underset{2}{\widetilde{N}}\overset{\left( b\right) }{_{\left( j\right) k}}%
\dfrac{\partial \widetilde{x}^{k}}{\partial x^{i}}=\underset{2}{N}\overset{%
\left( c\right) }{_{\left( k\right) i}}\dfrac{\partial \widetilde{t}^{b}}{%
\partial t^{c}}\dfrac{\partial x^{k}}{\partial \widetilde{x}^{j}}-\dfrac{%
\partial \widetilde{p}_{j}^{b}}{\partial x^{i}},%
\end{array}%
\end{equation*}%
is called a \textbf{nonlinear connection} on $E^{\ast }$. The components $%
\underset{1}{N}\overset{\left( a\right) }{_{\left( i\right) b}}$ (resp. $%
\underset{2}{N}\overset{\left( a\right) }{_{\left( i\right) j}}$) are called
the\ \textbf{temporal} (resp. \textbf{spatial}) \textbf{components} of $N$.
\end{definition}

Following now the geometrical ideas of Miron from \cite{Miro}, the paper 
\cite{Atan-Neag2} proves that any Kronecker $h$-regular multi-time
Hamiltonian function $H$ produces a natural nonlinear connection on the dual
1-jet space $E^{\ast }$, which depends only by the given Hamiltonian
function $H$:

\begin{theorem}
The pair of local functions $N=\left( \underset{1}{N}\text{{}}_{(i)b}^{(a)},%
\underset{2}{N}\text{{}}_{(i)j}^{(a)}\right) $ on $E^{\ast },$ where ($\chi
_{bc}^{a}$ are the Christoffel symbols of the semi-Riemannian temporal
metric $h_{ab}$)%
\begin{equation*}
\begin{array}{l}
\underset{1}{N}\text{{}}_{(i)b}^{(a)}=\chi _{bc}^{a}p_{i}^{c},\medskip \\ 
\underset{2}{N}\text{{}}_{(i)j}^{(a)}=\dfrac{h^{ab}}{4}\left[ \dfrac{%
\partial g_{ij}}{\partial x^{k}}\dfrac{\partial H}{\partial p_{k}^{b}}-%
\dfrac{\partial g_{ij}}{\partial p_{k}^{b}}\dfrac{\partial H}{\partial x^{k}}%
+g_{ik}\dfrac{\partial ^{2}H}{\partial x^{j}\partial p_{k}^{b}}+g_{jk}\dfrac{%
\partial ^{2}H}{\partial x^{i}\partial p_{k}^{b}}\right] ,%
\end{array}%
\end{equation*}%
represents a nonlinear connection on $E^{\ast },$ which is called the 
\textbf{canonical nonlinear connection of the multi-time Hamilton space }$%
MH_{m}^{n}=(E^{\ast },H).$
\end{theorem}

Taking into account the Theorem \ref{thchar} and using the \textit{%
generalized spatial Christoffel symbols} of the d-tensor $g_{ij}$, which are
given by%
\begin{equation*}
\Gamma _{ij}^{k}=\frac{g^{kl}}{2}\left( \frac{\partial g_{li}}{\partial x^{j}%
}+\frac{\partial g_{lj}}{\partial x^{i}}-\frac{\partial g_{ij}}{\partial
x^{l}}\right) ,
\end{equation*}%
we immediately obtain the following geometrical result:

\begin{corollary}
For $m=\dim \mathcal{T}\geq 2,$ the canonical nonlinear connection $N$ of a
multi-time Hamilton space $MH_{m}^{n}=(E^{\ast },H),$ whose Hamiltonian
function is given by (\ref{NEDTH}), has the components%
\begin{equation*}
\begin{array}{l}
\underset{1}{N}\text{{}}_{(i)b}^{(a)}=\chi _{bc}^{a}p_{i}^{c},\medskip \\ 
\underset{2}{N}\text{{}}_{(i)j}^{(a)}=-\Gamma
_{ij}^{k}p_{k}^{a}+T_{(i)j}^{(a)},%
\end{array}%
\end{equation*}%
where 
\begin{equation}
T_{(i)j}^{(a)}=\dfrac{h^{ab}}{4}\left( U_{ib\bullet j}+U_{jb\mathbf{\bullet }%
i}\right) ,  \label{aux_N2}
\end{equation}%
and%
\begin{equation*}
U_{ib}=g_{ik}U_{(b)}^{(k)},\qquad U_{kb\mathbf{\bullet }r}=\dfrac{\partial
U_{kb}}{\partial x^{r}}-U_{sb}\Gamma _{kr}^{s}.
\end{equation*}
\end{corollary}

\section{The Cartan canonical connection $C\Gamma (N)$ of a metrical
multi-time Hamilton space}

Let us consider that $MH_{m}^{n}=(J^{1\ast }(\mathcal{T},M),H)$ is a
multi-time Hamilton space, whose fundamental vertical metrical d-tensor is
given by (\ref{FVDT}). Let%
\begin{equation*}
N=\left( \underset{1}{N}\overset{\left( a\right) }{_{\left( i\right) b}},\ 
\underset{2}{N}\overset{\left( a\right) }{_{\left( i\right) j}}\right)
\end{equation*}%
be the canonical nonlinear connection of the multi-time Hamilton space $%
MH_{m}^{n}$.

\begin{theorem}[the generalized Cartan canonical $N$-linear connection]
On the multi-time Hamilton space $MH_{m}^{n}$ = $(J^{1\ast }(\mathcal{T}%
,M),H),$ endowed with the canonical nonlinear connection $N,$ there exists
an unique $h$-normal $N$-linear connection%
\begin{equation*}
C\Gamma (N)=\left( \chi _{bc}^{a},\text{ }A_{jc}^{i},\text{ }H_{jk}^{i},%
\text{ }C_{j\left( c\right) }^{i\left( k\right) }\right) ,
\end{equation*}%
having the metrical properties:\medskip

\emph{(i)} $g_{ij|k}=0,\quad g^{ij}|_{(c)}^{(k)}=0$,\medskip

\emph{(ii)} ${A_{jc}^{i}={\dfrac{g^{il}}{2}}{\dfrac{\delta g_{lj}}{\delta
t^{c}}},\quad H_{jk}^{i}=H_{kj}^{i},\quad C_{j(c)}^{i(k)}=C_{j(c)}^{k(i)},}$%
\medskip \newline
where \textquotedblright $_{/a}$\textquotedblright $,$ \textquotedblright $%
_{|k}$\textquotedblright\ and \textquotedblright $|_{(c)}^{(k)}$%
\textquotedblright\ represent the local covariant derivatives of the $h$%
-normal $N$-linear connection $C\Gamma (N)$.
\end{theorem}

\begin{proof}
Let $C\Gamma (N)=\left( \chi _{bc}^{a},\text{ }A_{jc}^{i},\text{ }H_{jk}^{i},%
\text{ }C_{j\left( c\right) }^{i\left( k\right) }\right) $ be an $h$-normal $%
N$-linear connection, whose local coefficients are defined by the relations%
\begin{equation*}
\begin{array}{l}
\medskip A_{bc}^{a}=\chi _{bc}^{a},\qquad A_{jc}^{i}={{\dfrac{g^{il}}{2}}{%
\dfrac{\delta g_{lj}}{\delta t^{c}},}} \\ 
\medskip H{_{jk}^{i}=}\dfrac{g^{ir}}{2}\left( {\dfrac{\delta g_{jr}}{\delta
x^{k}}}+{\dfrac{\delta g_{kr}}{\delta x^{j}}}-{\dfrac{\delta g_{jk}}{\delta
x^{r}}}\right) {,} \\ 
C_{i(c)}^{j(k)}{=-{\dfrac{g_{ir}}{2}}\left( {\dfrac{\partial g^{jr}}{%
\partial p_{k}^{c}}}+{\dfrac{\partial g^{kr}}{\partial p_{j}^{c}}}-{\dfrac{%
\partial g^{jk}}{\partial p_{r}^{c}}}\right) }.%
\end{array}%
\end{equation*}%
Taking into account the local expressions of the local covariant derivatives
induced by the $h$-normal $N$-linear connection $C\Gamma (N)$, by local
calculations, we deduce that $C\Gamma (N)$ satisfies conditions (i) and (ii).

Conversely, let us consider an $h$-normal $N$-linear connection%
\begin{equation*}
\tilde{C}\Gamma (N)=\left( \tilde{A}_{bc}^{a},\text{ }\tilde{A}_{jc}^{i},%
\text{ }\tilde{H}_{jk}^{i},\text{ }\tilde{C}_{j\left( c\right) }^{i\left(
k\right) }\right)
\end{equation*}%
which satisfies conditions (i) and (ii). It follows that we have%
\begin{equation*}
\tilde{A}_{bc}^{a}{=\chi _{bc}^{a},\quad \tilde{A}_{jc}^{i}={\frac{g^{il}}{2}%
}{\frac{\delta g_{lj}}{\delta t^{c}}}}.
\end{equation*}

Moreover, the metrical condition $g_{ij|k}=0$ is equivalent with%
\begin{equation*}
{\frac{\delta g_{ij}}{\delta x^{k}}}=g_{rj}\tilde{H}_{ik}^{r}+g_{ir}\tilde{H}%
_{jk}^{r}.
\end{equation*}%
Applying now a Christoffel process to indices $\{i,j,k\}$, we find%
\begin{equation*}
\tilde{H}{_{jk}^{i}={\frac{g^{ir}}{2}}\left( {\frac{\delta g_{jr}}{\delta
x^{k}}}+{\frac{\delta g_{kr}}{\delta x^{j}}}-{\frac{\delta g_{jk}}{\delta
x^{r}}}\right) }.
\end{equation*}

By analogy, using the relations $C_{j(c)}^{i(k)}=C_{j(c)}^{k(i)}$ and $%
g^{ij}|_{(c)}^{(k)}=0$, together with a Christoffel process applied to
indices $\{i,j,k\}$, we obtain%
\begin{equation*}
\widetilde{C}_{i(c)}^{j(k)}{=-{\dfrac{g_{ir}}{2}}\left( {\dfrac{\partial
g^{jr}}{\partial p_{k}^{c}}}+{\dfrac{\partial g^{kr}}{\partial p_{j}^{c}}}-{%
\dfrac{\partial g^{jk}}{\partial p_{r}^{c}}}\right) }.
\end{equation*}

In conclusion, the uniqueness of the \textit{generalized Cartan canonical
connection} $C\Gamma (N)$ on the dual $1$-jet space $E^{\ast }=J^{1\ast }(%
\mathcal{T},M)$ is clear.
\end{proof}

\begin{remark}
\emph{(i)} Replacing the canonical nonlinear connection N of the multi-time
Hamilton space $MH_{m}^{n}$ with an arbitrary nonlinear connection $\hat{N}$%
, the preceding Theorem holds good.

\emph{(ii)} The generalized Cartan canonical connection $C\Gamma (N)$ of the
multi-time Hamilton space $MH_{m}^{n}$ verifies also the metrical properties%
\begin{equation*}
h_{ab/c}=h_{ab|k}=h_{ab}|_{(c)}^{(k)}=0,\quad g_{ij/c}=0.
\end{equation*}

\emph{(iii)} In the case $m=\dim \mathcal{T}\geq 2$, the coefficients of the
generalized Cartan canonical connection $C\Gamma (N)$ of the multi-time
Hamilton space $MH_{m}^{n}$ reduce to%
\begin{equation}
A_{bc}^{a}=\chi _{bc}^{a},\quad A_{jc}^{i}={{\dfrac{g^{il}}{2}}{\dfrac{%
\partial g_{lj}}{\partial t^{c}}}},\quad H_{jk}^{i}=\Gamma _{jk}^{i},\quad {%
C_{j(c)}^{i(k)}}=0.  \label{Cartan-local-coeff}
\end{equation}
\end{remark}

\section{Local d-torsions and d-curvatures of the generalized Cartan
canonical connection $C\Gamma (N)$}

Applying the formulas that determine the local d-torsions and d-curvatures
of an $h$-normal $N$-linear connection $D\Gamma (N)$ (see these formulas in 
\cite{Oana+Neag}) to the generalized Cartan canonical connection $C\Gamma
(N) $, we obtain the following important geometrical results:

\begin{theorem}
The torsion tensor $\mathbb{T}$ of the generalized Cartan canonical
connection $C\Gamma (N)$ of the multi-time Hamilton space $MH_{m}^{n}$ is
determined by the local d-components%
\begin{equation*}
\begin{tabular}{|l|l|l|l|l|l|}
\hline
& $h_{\mathcal{T}}$ & \multicolumn{2}{|l|}{$h_{M}$} & \multicolumn{2}{|l|}{$%
v $} \\ \hline
& $m\geq 1$ & $m=1$ & $m\geq 2$ & $m=1$ & $m\geq 2$ \\ \hline
$h_{\mathcal{T}}h_{\mathcal{T}}$ & $0$ & $0$ & $0$ & $0$ & $R_{(r)ab}^{(f)}$
\\ \hline
$h_{M}h_{\mathcal{T}}$ & $0$ & $T_{1j}^{r}$ & $T_{aj}^{r}$ & $%
R_{(r)1j}^{(1)} $ & $R_{(r)aj}^{(f)}$ \\ \hline
$vh_{\mathcal{T}}$ & $0$ & $0$ & $0$ & $P_{(r)1(1)}^{(1)\;\;(j)}$ & $%
P_{(r)a(b)}^{(f)\;\;(j)}$ \\ \hline
$h_{M}h_{M}$ & $0$ & $0$ & $0$ & $R_{(r)ij}^{(1)}$ & $R_{(r)ij}^{(f)}$ \\ 
\hline
$vh_{M}$ & $0$ & $P_{i(1)}^{r(j)}$ & $0$ & $P_{(r)i(1)}^{(1)\;(j)}$ & $0$ \\ 
\hline
$vv$ & $0$ & $0$ & $0$ & $0$ & $0$ \\ \hline
\end{tabular}%
\end{equation*}%
where

\emph{(i)} for $m=\dim \mathcal{T}=1,$ we have%
\begin{equation*}
\begin{array}{c}
T_{1j}^{r}=-A_{j1}^{r},\quad P_{i(1)}^{r(j)}=C_{i(1)}^{r(j)},\quad
P_{(r)1(1)}^{(1)\;\ (j)}=\dfrac{\partial \underset{1}{N}\overset{\left(
1\right) }{_{\left( r\right) 1}}}{\partial p_{j}^{1}}+A_{r1}^{j}-\delta
_{r}^{j}\chi _{11}^{1},\medskip \\ 
{P_{(r)i(1)}^{(1)\;(j)}={\dfrac{\partial \underset{2}{N}\overset{\left(
1\right) }{_{\left( r\right) i}}}{\partial p_{j}^{1}}+}}H_{ri}^{j}{,}\quad
R_{(r)1j}^{(1)}={{\dfrac{\delta \underset{1}{N}\overset{\left( 1\right) }{%
_{\left( r\right) 1}}}{\delta x^{j}}}-{\dfrac{\delta \underset{2}{N}\overset{%
\left( 1\right) }{_{\left( r\right) j}}}{\delta t},}}\medskip \\ 
{{R_{(r)ij}^{(1)}}={{\dfrac{\delta \underset{2}{N}\overset{\left( 1\right) }{%
_{\left( r\right) i}}}{\delta x^{j}}}-{\dfrac{\delta \underset{2}{N}\overset{%
\left( 1\right) }{_{\left( r\right) j}}}{\delta x^{i}}}};}%
\end{array}%
\end{equation*}

\emph{(ii)} for $m=\dim \mathcal{T}\geq 2,$ using the equality (\ref{aux_N2}%
) and the notations%
\begin{equation*}
\begin{array}{l}
\medskip \chi _{fab}^{c}=\dfrac{\partial \chi _{fa}^{c}}{\partial t^{b}}-%
\dfrac{\partial \chi _{fb}^{c}}{\partial t^{a}}+\chi _{fa}^{d}\chi
_{db}^{c}-\chi _{fb}^{d}\chi _{da}^{c}, \\ 
\mathfrak{R}_{kij}^{r}=\dfrac{\partial \Gamma _{ki}^{r}}{\partial x^{j}}-%
\dfrac{\partial \Gamma _{kj}^{r}}{\partial x^{i}}+\Gamma _{ki}^{p}\Gamma
_{pj}^{r}-\Gamma _{kj}^{p}\Gamma _{pi}^{r},%
\end{array}%
\end{equation*}%
we have%
\begin{equation*}
\begin{array}{l}
\medskip T_{aj}^{r}=-A_{ja}^{r},\quad P_{(r)a(b)}^{(f)\;\;(j)}=\delta
_{b}^{f}A_{ra}^{j},\quad R_{(r)ab}^{(f)}=\chi _{gab}^{f}p_{r}^{g}, \\ 
\medskip R_{(r)aj}^{(f)}=-{\dfrac{\partial \underset{2}{N}\overset{\left(
f\right) }{_{\left( r\right) j}}}{\partial t^{a}}}-\chi _{ca}^{f}T{%
_{(r)j}^{(c)}}, \\ 
R_{(r)ij}^{(f)}=-\mathfrak{R}_{rij}^{k}p_{k}^{f}+\left[
T_{(r)i|j}^{(f)}-T_{(r)j|i}^{(f)}\right] .%
\end{array}%
\end{equation*}
\end{theorem}

\begin{theorem}
The curvature tensor $\mathbb{R}$ of the generalized Cartan canonical
connection $C\Gamma (N)$ of the multi-time Hamilton space $MH_{m}^{n}$ is
determined by the following adapted local curvature d-tensors:%
\begin{equation*}
\begin{tabular}{|l|l|l|l|l|l|}
\hline
& $h_{\mathcal{T}}$ & \multicolumn{2}{|l|}{$h_{M}$} & \multicolumn{2}{|l|}{$%
v $} \\ \hline
& $m\geq 1$ & $m=1$ & $m\geq 2$ & $m=1$ & $m\geq 2$ \\ \hline
$h_{\mathcal{T}}h_{\mathcal{T}}$ & $\chi _{abc}^{d}$ & $0$ & $R_{ibc}^{l}$ & 
$0$ & $-R_{(l)(a)bc}^{(d)(i)}$ \\ \hline
$h_{M}h_{\mathcal{T}}$ & $0$ & $R_{i1k}^{l}$ & $R_{ibk}^{l}$ & $%
-R_{(i)(1)1k}^{(1)(l)}=-R_{i1k}^{l}$ & $-R_{(l)(a)bk}^{(d)(i)}$ \\ \hline
$vh_{\mathcal{T}}$ & $0$ & $P_{i1(1)}^{l\;\;(k)}$ & $0$ & $%
-P_{(i)(1)1(1)}^{(1)(l)\;(k)}=-P_{i1(1)}^{l\;\;(k)}$ & $0$ \\ \hline
$h_{M}h_{M}$ & $0$ & $R_{ijk}^{l}$ & $\mathfrak{R}_{ijk}^{l}$ & $%
-R_{(i)(1)jk}^{(1)(l)}=-R_{ijk}^{l}$ & $-R_{(l)(a)jk}^{(d)(i)}$ \\ \hline
$vh_{M}$ & $0$ & $P_{ij(1)}^{l\;(k)}$ & $0$ & $-P_{(i)(1)j(1)}^{(1)(l)%
\;(k)}=-P_{ij(1)}^{l\;(k)}$ & $0$ \\ \hline
$vv$ & $0$ & $S_{i(1)(1)}^{l(j)(k)}$ & $0$ & $%
-S_{(i)(1)(1)(1)}^{(1)(l)(j)(k)}=-S_{i(1)(1)}^{l(j)(k)}$ & $0$ \\ \hline
\end{tabular}%
\end{equation*}%
where, for $m\geq 2$, we have the relations%
\begin{equation*}
-R_{(l)(a)bc}^{(d)(i)}=\delta _{l}^{i}\chi _{abc}^{d}-\delta
_{a}^{d}R_{lbc}^{i},\quad -R_{(l)(a)bk}^{(d)(i)}=-\delta
_{a}^{d}R_{lbk}^{i},\quad -R_{(i)(a)jk}^{(d)(l)}=-\delta _{a}^{d}\mathfrak{R}%
_{ijk}^{l},
\end{equation*}%
and, generally, the following formulas are true:

\emph{(i)} for $m=\dim \mathcal{T}=1,$ we have $\chi _{111}^{1}=0$ and%
\begin{equation*}
\begin{array}{l}
R_{i1k}^{l}=\dfrac{\delta A_{i1}^{l}}{\delta x^{k}}-\dfrac{\delta H_{ik}^{l}%
}{\delta t}+A_{i1}^{r}H_{rk}^{l}-H_{ik}^{r}A_{r1}^{l}+C_{i\left( 1\right)
}^{l\left( r\right) }R_{\left( r\right) 1k}^{\left( 1\right) },\medskip \\ 
R_{ijk}^{l}=\dfrac{\delta H_{ij}^{l}}{\delta x^{k}}-\dfrac{\delta H_{ik}^{l}%
}{\delta x^{j}}+H_{ij}^{r}H_{rk}^{l}-H_{ik}^{r}H_{rj}^{l}+C_{i\left(
1\right) }^{l\left( r\right) }R_{\left( r\right) jk}^{\left( 1\right)
},\medskip \\ 
P_{i1\left( 1\right) }^{l\ \left( k\right) }=\dfrac{\partial A_{i1}^{l}}{%
\partial p_{k}^{1}}-C_{i\left( 1\right) /1}^{l\left( k\right) }+C_{i\left(
1\right) }^{l\left( r\right) }P_{\left( r\right) 1\left( 1\right) }^{\left(
1\right) \ \left( k\right) },\medskip \\ 
P_{ij\left( 1\right) }^{l\ \left( k\right) }=\dfrac{\partial H_{ij}^{l}}{%
\partial p_{k}^{1}}-C_{i\left( 1\right) |j}^{l\left( k\right) }+C_{i\left(
1\right) }^{l\left( r\right) }P_{\left( r\right) j\left( 1\right) }^{\left(
1\right) \ \left( k\right) },\medskip \\ 
S_{i(1)(1)}^{l(j)(k)}=\dfrac{\partial C_{i(1)}^{l(j)}}{\partial p_{k}^{1}}-%
\dfrac{\partial C_{i(1)}^{l(k)}}{\partial p_{j}^{1}}%
+C_{i(1)}^{r(j)}C_{r(1)}^{l(k)}-C_{i(1)}^{r(k)}C_{r(1)}^{l(j)};%
\end{array}%
\end{equation*}

\emph{(ii)} for $m=\dim \mathcal{T}\geq 2,$ we have%
\begin{equation*}
\begin{array}{l}
\chi _{abc}^{d}=\dfrac{\partial \chi _{ab}^{d}}{\partial t^{c}}-\dfrac{%
\partial \chi _{ac}^{d}}{\partial t^{b}}+\chi _{ab}^{f}\chi _{fc}^{d}-\chi
_{ac}^{f}\chi _{fb}^{d},\medskip \\ 
R_{ibc}^{l}=\dfrac{\partial A_{ib}^{l}}{\partial t^{c}}-\dfrac{\partial
A_{ic}^{l}}{\partial t^{b}}+A_{ib}^{r}A_{rc}^{l}-A_{ic}^{r}A_{rb}^{l},%
\medskip \\ 
R_{ibk}^{l}=\dfrac{\partial A_{ib}^{l}}{\partial x^{k}}-\dfrac{\partial
\Gamma _{ik}^{l}}{\partial t^{b}}+A_{ib}^{r}\Gamma _{rk}^{l}-\Gamma
_{ik}^{r}A_{rb}^{l}{,}\medskip \\ 
\mathfrak{R}_{ijk}^{l}=\dfrac{\partial \Gamma _{ij}^{l}}{\partial x^{k}}-%
\dfrac{\partial \Gamma _{ik}^{l}}{\partial x^{j}}+\Gamma _{ij}^{r}\Gamma
_{rk}^{l}-\Gamma _{ik}^{r}\Gamma _{rj}^{l}.%
\end{array}%
\end{equation*}
\end{theorem}

\textbf{Acknowledgements.} The authors of this paper thank to Professor Gh.
Atanasiu for our interesting and useful discussions on this research topic.

Alexandru OAN\u{A} and Mircea NEAGU

University Transilvania of Bra\c{s}ov,

Department of Mathematics - Informatics,

Blvd. Iuliu Maniu, no. 50, Bra\c{s}ov 500091, Romania.

\textit{E-mails:} alexandru.oana@unitbv.ro, mircea.neagu@unitbv.ro

\end{document}